\documentclass[12pt]{amsart}
\usepackage{amssymb}
\usepackage{amsmath}
\usepackage{pdfpages}
\usepackage{graphicx} 
\usepackage{mathabx}
\usepackage{amsthm}
\usepackage{enumerate}
\usepackage[all,arc,curve,color,frame]{xy}
\usepackage[utf8]{inputenc}

\DeclareMathOperator{\End}{End}

\DeclareMathOperator{\tr}{tr}
\DeclareMathOperator{\ad}{ad}

\newcommand{\A}{{\mathcal A}}
\newcommand{\B}{{\mathcal B}}
\newcommand{\calC}{{\mathcal C}}
\newcommand{\E}{{\mathcal E}}
\newcommand{\F}{{\mathcal F}}

\newcommand{\G}{{\mathcal G}}
\newcommand{\I}{{\mathcal I}}
\newcommand{\J}{{\mathcal J}}

\newcommand{\K}{{\mathcal K}}

\renewcommand{\H}{{\mathcal H}}

\renewcommand{\P}{{\mathcal P}}

\renewcommand{\ll}{{\mathbf l}}
\newcommand{\rr}{{\mathbf r}}
\newcommand{\N}{{\mathbb N}}
\newcommand{\Ncal}{{\mathcal N}}
\newcommand{\C}{\ensuremath{\mathbb{C}}}
\newcommand{\R}{\ensuremath{\mathbb{R}}}

\newcommand{\p}{\partial}
\newcommand{\s}{{\rm Symb}}

\newcommand{\D}{{\mathcal D}}
\newcommand{\unit}{{\bf{1}}}

\newcommand{\U}{{\mathcal U}}
\newcommand{\V}{{\mathcal V}}
\newcommand{\W}{{\mathcal W}}
\newcommand{\x}{{x_0}}
\newcommand{\X}{{\mathcal X}}

\newcommand{\Z}{\ensuremath{\mathbb{Z}}}
\newcommand{\lk}{{\langle z \rangle}}
\newcommand{\rk}{{\langle \bar z \rangle}}
\newcommand{\ba}{\begin{eqnarray*}}
\newcommand{\ea}{\end{eqnarray*}}
\newtheorem{definition}{Definition}[section]
\newtheorem{lemma}{Lemma}[section]
\newtheorem{proposition}{Proposition}[section]
\newtheorem{theorem}{Theorem}[section]
\newtheorem{corollary}{Corollary}[section]
\begin{document}

\title[The algebra of distributions]
{An algebra of distributions related to a star product with separation of variables}
\author[Alexander Karabegov]{Alexander Karabegov}
\address[Alexander Karabegov]{Department of Mathematics, Abilene
Christian University, ACU Box 28012, Abilene, TX 79699-8012}
\email{axk02d@acu.edu}


\maketitle

\begin{abstract}
Given a star product with separation of variables $\star$ on a pseudo-K\"ahler manifold $M$ and a point $\x \in M$, we construct an associative algebra of formal distributions supported at $\x$. We use this algebra to express the formal oscillatory exponents of a family of formal oscillatory integrals related to the star product~$\star$.
\end{abstract}

\section{Introduction}

A classical mechanical system can be described by the Poisson algebra $C^\infty(M)$ of complex-valued functions on a Poisson manifold $M$ with two operations, the pointwise product $f,g \mapsto fg$ and the Poisson bracket $f,g \mapsto \{f,g\}$. A quantum mechanical system can be described by an algebra of operators acting on a Hilbert space $H$. The corresponding two operations are the noncommutative product of operators, $A,B \mapsto AB$, and the commutator $A,B \mapsto [A,B]=AB-BA$. If $M$ is a symplectic manifold with the Poisson structure induced by the symplectic structure, a quantization of the classical mechanical system on $M$ is given by a family of Hilbert spaces $\{H_h\}$ parametrized by a small parameter $h$ and a family of operator algebras acting on the respective spaces $H_h$ which exhibits a semiclassical behavior and approximates the Poisson algebra $(C^\infty(M), \{\cdot,\cdot\})$ as $h \to 0$. 

There are two different formalisms relating classical and quantum systems in terms of operator symbols. In one approach, one considers for each $h$ a noncommutative algebra $\A_h \subset C^\infty(M)$ with product $\ast_h$ and a representation $f \mapsto \mathrm{Op}_h (f)$ of $\A_h$ in $H_h$.  The Correspondence Principle requires that
\[
\lim_{h \to 0} f \ast_h g = fg \mbox{ and } \lim_{h \to 0} h^{-1}(f \ast_h g - g \ast_h f) = \{f,g\}.
\footnote{One usually assumes that $h$ is purely imaginary or inserts the imaginary unit $i$ in front of  $\{f,g\}$ if $h$ is assumed to be real.}
\]
In this approach, $f$ is called a symbol of the operator $\mathrm{Op}_h(f)$. In a stronger form, one requires the existence of a full asymptotic expansion of the composition of symbols as $h \to 0$,
\[
f \ast_h g \sim fg + h C_1(f,g) + h^2 C_2(f,g)+ \ldots,
\]
such that the operators $C_r$ are bidifferential, $C_1(f,g)-C_1(g,f) = \{f,g\}$, and there is an associative product $\star$ on the space $C^\infty(M)[[\nu]]$ of formal series in the formal parameter $\nu$ given by the formula
\begin{equation}\label{E:star}
f \star g = fg + \sum_{r=1}^\infty \nu^r C_r(f,g).
\end{equation}
The product $\star$ on $M$ is called a star product. The concept of a star product on a Poisson manifold was introduced in \cite{BFFLS}. 

Using this approach, one obtains eponymous star products from $pq$-,\\
 $qp$-, Wick, anti-Wick, and Weyl operator symbols and various constructions of star products on cotangent bundles (see \cite{BNW}, \cite{P}). 

In the other approach, one considers two symbol mappings. One maps a function $f$ on $M$ to an operator $\mathrm{Op}_h(f)$ on $H_h$ and $f$ is called a contravariant symbol of $\mathrm{Op}_h(f)$. The other one maps an operator $A$ on $H_h$ to a function $S_h(A)$ on $M$ called the covariant symbol of $A$. In this approach the existence of composition of symbols is not required and the star product is extracted from the following two mappings. The mapping
\[
B_hf = S_h (\mathrm{Op}_h(f))
\]
which maps a contravariant symbol to the corresponding covariant symbol is called the Berezin transform. The mapping
\[
Q_h(f,g) = S_h(\mathrm{Op}_h(f) \mathrm{Op}_h(g))
\]
is a twisted product of contravariant symbols which is not associative. It is assumed that there exist the asymptotic expansions
\ba
B_h f \sim f + h B_1f + h^2 B_2f + \ldots \quad \mbox{ and } \\
Q_h(f,g) \sim fg + hQ_1(f,g)+ h^2 Q_2(f,g)+\ldots,
\ea
as $h \to 0$ such that $B_r$ are differential operators and $Q_r$ are bidifferential operators on $M$ and the formal operators
\[
Bf = f + \sum_{r=1}^\infty \nu^r B_r f\mbox{ and } Q(f,g) = fg + \sum_{r=1}^\infty \nu^r Q_r(f,g)
\]
give rise to two star products,
\begin{equation}\label{E:twostar}
f \star g = Q(B^{-1}f, B^{-1}g) \mbox{ and } f \star'  g = B^{-1}Q(f,g).
\end{equation}
The operator $B = 1 + \nu B_1 + \nu^2 B_2 + \ldots$ is called the formal Berezin transform. In general, the Berezin transform $B_h$ is not invertible, but, by assumption, $B$ is invertible. This formalism (sometimes in a different but equivalent form) is used in the construction of the Berezin and Berezin-Toeplitz star products on K\"ahler manifolds and on general symplectic manifolds (see \cite{BMS}, \cite {E}, \cite{KS}, \cite{G}, \cite{Ch}). 

It should be noted that the most important constructions of star products of Fedosov \cite{F1} and Kontsevich \cite{K} do not use operator symbols.

The formulas for composition of some types of symbols, the Berezin transforms, and twisted products are given by oscillatory or Laplace-type integrals depending on $h$ which have asymptotic expansions as $h \to 0$. Such asymptotic expansions can be described in terms of formal oscillatory integrals (FOIs) which are given by formal oscillatory integral kernels (see details in Section \ref{S:foi}). Such kernels were explicitly calculated for the Berezin-Toeplitz star product on an arbitrary compact K\"ahler manifold and were instrumental in the complete identification of this star product in \cite{KS}. We expect that the information encoded in the formal oscillatory integral kernels can be used to describe and parametrize special classes of star products, in particular the Berezin-Toeplitz star product on general symplectic manifolds.

In this paper we develop tools allowing to express formal oscillatory integral kernels related to a star product in terms of that star product. We apply these tools to the well-understood case of the Berezin and Berezin-Toeplitz star products on K\"ahler manifolds hoping to apply them to more general star products in the future work.

In Section \ref{S:foi} we define formal oscillatory integrals and describe their main properties. In Section \ref{S:star} we recall basic facts on star products. In Section \ref{S:sep} we review the star products with separation of variables on pseudo-K\"ahler manifolds. In Sections \ref{S:algb} and \ref{S:algc} we construct two auxiliary algebras that are used in Section \ref{S:algd} to construct the algebra of distributions from the title of this paper. Then we give a formula which expresses the formal oscillatory exponents of a family of formal oscillatory integrals related to a star product with separation of variables in terms of this algebra of distributions. We expect that an analogous formula will be valid for the Berezin and Berezin-Toeplitz star products on general symplectic manifolds.

\section{Formal oscillatory integrals}\label{S:foi}

Given a vector space $V$, we denote by $V((\nu))$ the space of $\nu$-formal vectors
\begin{equation}\label{E:fvect}
    v =  \nu^r v_r + \nu^{r+1} v_{r+1} + \ldots,
\end{equation}
where $r \in \Z$ and $v_k \in V$ for all $k \geq r$. The subspace $V[[\nu]] \subset V((\nu))$ consists of the vectors (\ref{E:fvect}) with $r =0$.

In this paper we use the $\nu$-formal extensions of the vector spaces of functions, jets of infinite order of functions at a given point, differential operators, and distributions supported at a point. 

Given a manifold $M$ of dimension $n$ and a point $\x\in M$, we denote by $\J_\x$ the space of jets of infinite order at $\x$. It is defined as follows. Consider the vanishing ideal $\I_\x=\{ f \in C^\infty(M)| f(\x)=0\}$ of the point $\x$ in the algebra $C^\infty(M)$. Then
\[
\J_\x := C^\infty(M) / \bigcap\displaylimits_{r=1}^\infty \I_\x^r.
\]
The jet of a function $f\in C^\infty(M)$ at $x$ is the coset $f+\bigcap\displaylimits_{r=1}^\infty \I_\x^r$. In local coordinates $\{x^i\}$ around $\x$ (so that $x^i(\x)=0$), $\J_\x$ can be identified with $\C[[x^1, \ldots, x^n]]$. Namely, 
the jet of infinite order of a function $f$ at $\x$ is identified with the Maclaurin series of $f$, which is a formal series in the variables $x^i$. All jets in this paper are of infinite order.

Formal oscillatory integrals (FOIs) were introduced in \cite{KS} and developed further in \cite{LMP10}, where it was shown  that there exists a formal algebraic counterpart of an oscillatory integral with a complex phase function on a manifold $M$. 

Let $\x\in M$ be a fixed point, $\varphi = \nu^{-1} \varphi_{-1} + \varphi_0 + \ldots$ be a $\nu$-formal complex phase function on~$M$ such that $\x$ is a nondegenerate critical point of $\varphi_{-1}$ with zero critical value, $\varphi_{-1}(\x)=0$, and $\rho = \rho_0 + \nu \rho_1 + \ldots$ be a $\nu$-formal complex volume form (density) on $M$ such that $\rho_0$ does not vanish at $\x$. We call $(\varphi, \rho)$ {\it a phase-density pair at $\x$}. We will be interested only in the jets of $\varphi$ and $\rho$ at $\x$. Two pairs, $(\varphi, \rho)$ and $(\hat\varphi, \hat\rho)$, at $\x$ are called equivalent if there exists a formal function $u = u_0 + \nu u_1 + \ldots$ on a neighborhood of $\x$ such that
\[
    \hat\varphi = \varphi+ u \mbox{ and } \hat \rho = e^{-u} \rho. 
\]
We can write the equivalence class of pairs $(\varphi, \rho)$ as
\begin{equation}\label{E:fosck}
   e^{\varphi}  \rho.
\end{equation}

\begin{definition}\label{D:FOI}
Given a pair $(\varphi, \rho)$ on a manifold $M$ at $\x \in M$, a formal distribution $\Lambda = \Lambda_0 + \nu \Lambda_1 + \ldots$ on $M$ supported at $\x$ is called a formal oscillatory integral (FOI) associated with the pair $(\varphi, \rho)$ if $\Lambda_0 \neq 0$ and 
  \begin{equation}\label{E:FOI}
       \Lambda\left(vf + (v\varphi + \mathrm{div}_\rho v) f\right)=0
  \end{equation}
 for any vector field $v$ and any function $f$ on $M$.
\end{definition}
Here $\mathrm{div}_\rho v = \mathbb{L}_v \rho/\rho$ is the divergence of the vector field $v$ with respect to $\rho$ and $\mathbb{L}_v$ is the Lie derivative with respect to $v$. 

\medskip

{\it Example} If $\psi$ is a real phase function on $\R^n$ with a nondegenerate critical point $\x$ with zero critical value, $\psi(\x)=0$, $f$ is an amplitude supported on a small neighborhood of $\x$, and $h$ is a purely imaginary numerical parameter such that $ih>0$, then, according to the stationary phase method, there is an asymptotic expansion
\[
h^{-\frac{n}{2}} \int_{\R^n} e^{h^{-1} \psi(x)} f(x)\, dx \sim \Lambda_0(f)  + h\Lambda_1(f) + h^2 \Lambda_2(f) + \ldots
\]
as $h \to 0$, where $\Lambda_r, r \geq 0$, are distributions supported at $\x$ (see \cite{L}). The formal distribution $\Lambda = \Lambda_0+\nu\Lambda_1+\nu^2\Lambda_2+\ldots$ is a FOI at $\x$ associated with the pair $(\nu^{-1}\psi, dx)$.

\medskip

Heuristically, the $\nu$-formal distribution $\Lambda$ in Definition \ref{D:FOI} gives an interpretation of the formal expression
 \begin{equation}\label{E:Lie}
    \Lambda(f) = \nu^{-\frac{n}{2}} \int_M e^\varphi f\, \rho,
 \end{equation}
where $n = \dim M$ and $f$ is an amplitude.  Condition (\ref{E:FOI}) corresponds to the formal property of (\ref{E:Lie}) that
\[
\nu^{-\frac{n}{2}} \int_M \mathbb{L}_v(e^\varphi f\, \rho)=0.
\]
As shown in \cite{LMP10},
\[
       \Lambda_0 = \alpha \delta_\x,
\]
where $\alpha$ is a nonzero complex constant and $\delta_\x$ is the Dirac distribution at $\x$, $\delta_\x(f)=f(\x)$, which agrees with the stationary phase lemma. For any pair $(\varphi, \rho)$ there exists an associated FOI which is determined up to a formal multiplicative complex constant $c(\nu) = c_0 + \nu c_1 + \ldots$ with $c_0\neq 0$. In particular, there is a unique such FOI $\Lambda$ for which $\Lambda(1)=1$. If a FOI is associated with a pair $(\varphi, \rho)$, then it is associated with any equivalent pair. Thus, $\Lambda$ is associated with the oscillatory kernel (\ref{E:fosck}).

Given a pair $(\varphi, \rho)$ and a $\nu$-formal volume form $\hat\rho=\hat\rho_0 + \nu \hat\rho_1 + \ldots$ such that $\hat\rho_0$ does not vanish at $\x$, there exists a formal phase function $\hat\varphi$ such that the pairs $(\varphi, \rho)$ and $(\hat\varphi, \hat\rho)$ are equivalent. Thus, to compare two equivalence classes of phase-density pairs, one can assume that they share the same density.

It is clear from the definition that a FOI at $\x$ associated with a pair $(\varphi, \rho)$ depends only on the jets  of $\varphi$ and $\rho$ at $\x$. It was shown in \cite{LMP10} that if a FOI $\Lambda$ at $\x$ is associated with pairs $(\varphi, \rho)$ and $(\hat\varphi, \rho)$ with the same volume form $\rho$, then the jet of $\hat\varphi-\varphi$ at $\x$ is a $\nu$-formal constant. This result is based on the following important observation. Given a FOI $\Lambda$ at $\x$, consider a pairing on $C^\infty(M)[[\nu]]$ given by the formula
 \begin{equation}\label{E:pairl}
    (f, g)_{\Lambda} :=\Lambda(f\cdot g).
 \end{equation}
This pairing depends only on the jets of $f$ and $g$ at $\x$ and therefore it induces a pairing on the space $\J_\x[[\nu]]$ of $\nu$-formal jets at $\x$. The following statement was proved in \cite{LMP10}:
\begin{lemma}\label{L:nondeg}
The induced pairing on $\J_\x[[\nu]]$ is nondegenerate. 
\end{lemma}

\section{General properties of star products}\label{S:star}

Let $\star$ be any star product on a Poisson manifold $M$. Since $\star$ is given by bidifferential operators, it can be restricted to any open subset of~$M$. 

We assume that the unit constant 1 is the unity of $\star$, so that $f \star 1 = f = 1 \star f$ for any $f$.

We denote by $L_f$ the left star-multiplication operator by a function $f$ and by $R_g$ the right star-multiplication operator by $g$. Then $L_f g = f \star g = R_g f$. The associativity of $\star$ is equivalent to the requirement that $L_f$ commute with $R_g$ for all $f,g$.

Star products $\star_1$ and $\star_2$ on $M$ are called equivalent if there exists a formal differential operator $T=1+\nu T_1 + \nu^2 T_2 + \ldots$ on $M$ such that $Tf \star_1 Tg=T(f \star_2 g)$.

\medskip

{\it Example} The star products $\star$ and $\star^\prime$ in (\ref{E:twostar}) are equivalent with the formal Berezin transform $B$ being an equivalence operator.

\medskip

Kontsevich constructed a star product on $\R^n$ equipped with an arbitrary Poisson structure in \cite{K}. He showed that star products exist on arbitrary Poisson manifolds and gave an explicit parametrization of their equivalence classes. On symplectic manifolds Fedosov constructed star products in each equivalence class in \cite{F1} and \cite{F2}. 

Let $M$ be a symplectic manifold with symplectic form $\omega_{-1}$ and $\star$ be any star product on $M$.  There exists a globally defined $\nu$-formal density $\mu$ on $M$ with the trace property
\[
\int_M f \star g \, \mu = \int_M g \star f \, \mu
\]
for any functions $f$ and $g$ such that $fg$ is compactly supported (see, say, \cite{LMP3} and \cite{GR2}). If $M$ is connected, the global trace density is unique up to a nonzero formal multiplicative constant. There exists a canonical normalization of the trace density which is used in the statement of the algebraic index theorem first proved in \cite{F2}.

A star product (\ref{E:star}) on a Poisson manifold is called natural in \cite{GR} if for each $r\geq 1$ the bidifferential operator $C_r$ is of order not greater than $r$ in both arguments. All classical star products are natural. The natural star products are strongly related to the Lagrangian asymptotic analysis and have many important properties. To describe these properties, we need several definitions.

\begin{definition}
A formal differential operator $X=X_0+\nu X_1 + \ldots$ on a manifold $M$ is called natural if for every $r\geq 0$ the differential operator $X_r$ is of order not greater than $r$. A formal differential operator $A=A_0+\nu A_1 + \ldots$ on $M$ is called oscillatory if $A= \exp \nu^{-1}X$, where $X$ is a natural operator such that $X_0=X_1=0$. Then $A_0=1$ and $A_1=X_2$.
\end{definition}

Denote by $\mathfrak{N}$ the space of natural operators on~$M$. It is an associative algebra. It is also a Lie algebra with the operation $A, B \mapsto \nu^{-1}[A,B]$. Alternatively, $\nu^{-1}\mathfrak{N}$ is a Lie algebra with respect to the usual commutator $A,B \mapsto [A,B]$. 

A star product $\star$ is natural if and only if the operators $L_f$ and $R_f$ are natural for any $f=f_0 + \nu f_1 +\ldots$. The following important theorem was proved by Gutt and Rawnsley in \cite{GR}:
\begin{theorem}\label{T:GR}
If $\star_1$ and $\star_2$ are two equivalent natural star products on a Poisson manifold $M$, then any equivalence operator $T= 1 + \nu T_1 + \ldots$ such that $Tf \star_1 Tg =T(f \star_2 g)$ is oscillatory.
\end{theorem}

Any distribution $\Lambda$ on a manifold $M$ supported at a point $\x\in M$ can be represented as $\Lambda=\delta_\x \circ A$ for some differential operator $A$ on~$M$. The order of $\Lambda$ is the smallest $k$ such that $\Lambda(f)=0$ for any function with zero of order $k+1$ at $\x$.

\begin{definition}
A formal distribution $\Lambda=\Lambda_0 + \nu \Lambda_1 + \ldots$ on a manifold $M$ supported at a point $\x\in M$ is called natural if for every $r\geq 0$ the order of the distribution $\Lambda_r$ is not greater than $r$. A formal distribution $\Lambda=\Lambda_0 + \nu \Lambda_1 + \ldots$ on $M$ is called oscillatory if there exists an oscillatory operator $A$ on $M$ such that $\Lambda = \delta_\x \circ A$.
\end{definition}

For any natural operator $X$ on $M$ the formal distribution $\Lambda= \delta_\x \circ X$ is natural. For any natural distribution $\Lambda$ supported at $\x$ there exists a natural operator $X$ such that $\Lambda= \delta_\x \circ X$.

It was proved in \cite{Asympt} that a star product $\star$ on $M$ is natural if and only if the distribution $f \otimes g \mapsto (f \star g)(x)$ on $M \times M$ supported at $(x,x)$ is oscillatory for every $x\in M$.

If $\Lambda=\Lambda_0 + \nu \Lambda_1 + \ldots$ is an oscillatory distribution on $M$ supported at $\x\in M$, then $\Lambda_0=\delta_\x$ and $\Lambda_1$ is a distribution of order at most 2. There exists a unique symmetric bilinear form $\beta_\Lambda$ on $T^\ast_\x M$ such that
\[
\beta_\Lambda(df(\x), dg(\x))=\Lambda_1(fg)
\]
for any functions $f$ and $g$ that vanish at $\x$.

\begin{definition}
An oscillatory distribution $\Lambda$ is called nondegenerate if the form $\beta_\Lambda$ is nondegenerate.
\end{definition}

It was proved in \cite{Asympt} that a formal distribution $\Lambda=\Lambda_0 + \nu \Lambda_1 + \ldots$ on a manifold $M$ supported at a point $\x\in M$ is a FOI if and only if $\Lambda$ is a nondegenerate oscillatory distribution.

If (\ref{E:star}) is a natural star product on a Poisson manifold $M$, then in local coordinates $\{x^i\}$ we have
\begin{equation}\label{E:tensK}
C_1(f,g)= K^{ij} \frac{\p f}{\p x^i}\frac{\p g}{\p x^j},
\end{equation}
where $K^{ij}$ is a tensor such that $K^{ij}-K^{ji}$ is the Poisson tensor on $M$. Given $x\in M$,  the component $\Lambda_1$ of the oscillatory distribution $\Lambda(f \otimes g) = (f \star g)(x)$ on $M \times M$ supported at $(x,x)$ is
\[
\Lambda_1(F)= K^{ij}(x) \frac{\p^2 F}{\p x^i \p y^j}(x,x),
\]
where $F$ is a function on $M \times M$. If $F$ and $G$ are functions on $M \times M$ such that $F(x,x)=0$ and $G(x,x)=0$, we have that
\ba
\beta_{\Lambda}(dF(x,x), dG(x,x)) = K^{ij}(x) \frac{\p^2 FG}{\p x^i \p y^j}(x,x)=\\
K^{ij}(x) \left(\frac{\p F}{\p x^i}(x,x) \frac{\p G}{\p y^j}(x,x) + \frac{\p F}{\p y^j}(x,x) \frac{\p G}{\p x^i}(x,x)\right).
\ea
The form $\beta_\Lambda$ is therefore given by the symmetric anti-diagonal block matrix
\[
\begin{bmatrix}
0 & K(x)\\
K^t(x) & 0
\end{bmatrix}
\]
which is nondegenerate if and only if the tensor $K^{ij}(x)$ is nondegenerate. We have thus proved the following statement.
\begin{proposition}
Given a natural star product $\star$ on a manifold $M$ and a point $x\in M$, the distribution 
\[
\Lambda(f \otimes g) = (f \star g)(x)
\]
on $M \times M$ supported at $(x,x)$ is a formal oscillatory integral if and only if the tensor $K^{ij}(x)$ in (\ref{E:tensK}) is nondegenerate.
\end{proposition}
In particular, if $M$ is symplectic and $C_1(f,g) = \frac{1}{2}\{f,g\}$ is skew symmetric, then  $K^{ij}$ is a nondegenerate Poisson tensor. It follows that Fedosov's star products are given by formal oscillatory integrals.

We plan to study the oscillatory kernels of such FOIs and use them for a geometric classification of star products in our future work.

In this paper we will consider the opposite case of natural star products on a symplectic manifold $M$ for which the rank of the tensor $K^{ij}$ is $\frac{1}{2}\dim M$.

\section{Star products with separation of variables}\label{S:sep}

Berezin described in \cite{Ber1} and \cite{Ber2} a quantization procedure on K\"ahler manifolds which leads to star products with the property of separation of variables (see, e.g., \cite{BW}, \cite{CGR}, \cite{E}, \cite{CMP1}, \cite{KS}). It is natural to consider the star products with this property on pseudo-K\"ahler manifolds. Recall that a pseudo-K\"ahler manifold is a complex manifold equipped with a real symplectic form of type $(1,1)$ with respect to the complex structure.

\begin{definition}\label{D:sep}
A star product (\ref{E:star}) on a pseudo-K\"ahler manifold $M$ is called a star product with separation of variables if the operators $C_r, r \geq 1,$ differentiate the first argument in holomorphic directions and the second argument in antiholomorphic ones. 
\end{definition}

The simplest example of such star product is the anti-Wick star product on $\C^n$,

\[
f \star g = \sum\displaylimits_{r=0}^\infty \frac{\nu^r}{r!} \sum_{k_1, \ldots,k_r} \frac{\p^r f}{\p z^{k_1} \ldots \p z^{k_r}} \frac{\p^r g}{\p \bar z^{k_1} \ldots \p \bar z^{k_r}}.
\]

As shown by Astashkevich in \cite{Ast}, the concept of a star product with separation of variables can be generalized to symplectic manifolds equipped with two transversal complex Lagrangian polarizations. In Definition \ref{D:sep} these are the holomorphic and the anti-holomorphic polarizations. It is natural to call the star products in Definition \ref{D:sep} the star products of the anti-Wick type. If the roles of holomorphic and antiholomorphic derivatives in Definition \ref{D:sep} are switched, the corresponding star products are called the star products of the Wick type, as introduced in \cite{BW}.

Let $\star$ be a product of the anti-Wick type on $M$. If $a$ is a holomorphic function and $b$ is an antiholomorphic function locally defined on $M$, then for any function $f$ we have
\[
     a \star f = af \mbox{ and } f \star b = bf,
\]
i.e., $L_a = a$ and $R_b = b$ are pointwise multiplication operators.

We will denote the pointwise multiplication operator by a function $f$ by the same symbol $f$ throughout this paper.
 
 Let $(M,\omega_{-1})$ be a pseudo-K\"ahler manifold. In \cite{CMP1} it was shown that the star products of the anti-Wick type on $M$ are bijectively parametrized (not only up to equivalence) by the formal closed (1,1)-forms
\begin{equation}\label{E:omega}
    \omega = \nu^{-1}\omega_{-1} + \omega_0 + \nu \omega_1 + \ldots
\end{equation}
on $M$. We will briefly recall this parametrization. A closed $(1,1)$-form $\alpha$ on a Stein neighborhood has a potential $\varphi$ such that $\alpha=i\p \bar \p \varphi$.

For any form (\ref{E:omega}) on $(M, \omega_{-1})$ there exists a unique star product of the anti-Wick type $\star$ on $M$ such that on every Stein coordinate chart $U \subset M$ and for any $\nu$-formal potential $\Phi$ of $\omega$ on $U$ the following operators are explicitly given:
\[
    L_{\frac{\p\Phi}{\p z^k}} = \frac{\p\Phi}{\p z^k} + \frac{\p}{\p z^k} \mbox{ and } R_{\frac{\p\Phi}{\p \bar z^l}} = \frac{\p\Phi}{\p \bar z^l} + \frac{\p}{\p \bar z^l}.
\]
The formal form $\omega$ is called {\it the classifying form of the star product~$\star$.} Every star product of the anti-Wick type has a unique classifying form.

Given a star product $\star$ of the anti-Wick type on $M$, there exists a $\nu$-formal differential operator
\[
    B = 1 + \nu B_1 + \nu^2 B_2 + \ldots
\]
globally defined on $M$ and such that for any local holomorphic function $a$ and local antiholomorphic function $b$,
\[
     B(ab) = b \star a.
\]
It is called the formal Berezin transform of the star product $\star$. Observe that $Ba=a$ and $Bb = b$. It is proved in \cite{CMP3} that
\begin{equation}\label{E:iaib}
    L_b = B\circ b \circ B^{-1} \mbox{ and } R_a=B \circ a \circ B^{-1}.
\end{equation}
One can recover the product $\star$ from the operator $B$ using that
\[
    (ab) \star (a^\prime  b^\prime ) = a B(a^\prime b) b^\prime ,
\]
where the functions $a,a'$ are local holomorphic and $b,b^\prime $ are local antiholomorphic. The equivalent star product
\begin{equation}\label{E:wicktype}
     f \star^\prime  g := B^{-1} (Bf \star Bg)
\end{equation} 
on $M$ is a star product with separation of variables {\it of the Wick type}.

\begin{lemma}\label{L:bertr}
  The formal Berezin transform $B$ of a star product of the anti-Wick type $\star$ is oscillatory.
\end{lemma}
\begin{proof}
The star products  $\star$ and $\star^\prime$ are natural (see  \cite{N}). Since $B$ is an equivalence operator between the products $\star$ and $\star'$, it is oscillatory according to Theorem \ref{T:GR}.
\end{proof}

It was shown in \cite{KS} and \cite{LMP10} that for any point $\x \in M$ and any integer $l \geq 1$ the functional
\[
    K^{(l)}(f_1, \ldots, f_l) = B(f_1 \star' \ldots \star' f_l)(\x) = (Bf_1 \star \ldots \star Bf_l)(\x)
\]
on $M^l$ is a FOI at $(\x)^l=(\x,\ldots,\x)  \in M^l$. Below we give a phase-density pair associated with $K^{(l)}$, which was found in \cite{KS} and \cite{LMP10}. 

If $N$ is an embedded (regular) submanifold of a manifold $M$ and $\I_N = \{f \in C^\infty(M)| \ f |_N=0\}$ is the vanishing ideal of $N$ in the algebra $C^\infty(M)$, we denote by $C^\infty(M,N)$ the space of jets along $N$,
\[
C^\infty(M,N) = C^\infty(M) / \bigcap_{r=1}^\infty \I^r_N.
\]
We think of $C^\infty(M,N)$ as of the algebra of functions on the formal neighborhood $(M,N)$ of $N$ in $M$.

Let $U$ be a Stein neighborhood in $M$ and $\Phi$ be a potential of the classifying form $\omega$ of the product $\star$ on $U$. Let $\widebar U$ denote a copy of $U$ equipped with the opposite complex structure. One can find a function 
$\tilde\Phi(x,y)$ on $U \times \widebar U$ such that $\tilde\Phi (x,x)= \Phi(x)$ and
\[
     \bar\p_{U \times \widebar U} \tilde{\Phi}
\]
has zero of infinite order at every point of the diagonal of $U \times \widebar U$. The function $\tilde\Phi(x,y)$ is called an almost analytic extension of  $\Phi$  (see details in \cite{LMP10}). For each $l \geq 1$ we introduce a function $G^{(l)}$ on $U^l$ by the formula
\begin{eqnarray*}
     G^{(l)}(x_1, \ldots x_l) := \tilde\Phi(x_1,x_2) +  \tilde\Phi(x_2,x_3) + \ldots  + \tilde \Phi(x_l,x_1)   \\
     - (\Phi(x_1) + \Phi(x_2) + \ldots + \Phi(x_l)).
\end{eqnarray*}
This function defines an element of $\nu^{-1}C^\infty(U^l,U)[[\nu]]$, where $U$ is identified with the diagonal of $U^l$. This element does not depend on the choice of the potential $\Phi$ and of the almost analytic extension of $\Phi$. Thus, taking such functions for every Stein neighborhood in $M$, we get a global element of $\nu^{-1}C^\infty(M^l,M)[[\nu]]$. We call it the {\it cyclic formal $l$-point Calabi function of the classifying form $\omega$.}

Now suppose that a point $\x\in U$ is fixed and consider the function
\begin{equation}\label{E:fgl}
F^{(l)}(x_1,\ldots, x_l) := G^{(l+1)}(\x, x_1, \ldots, x_l)
\end{equation}
on $U^l$. The jet of $F^{(l)}$ at $(\x)^l =(\x,\ldots,\x)\in U^l$ is determined by the jet of $G^{(l+1)}$ at $(\x)^{l+1}\in U^{l+1}$. It was shown in \cite{KS} and \cite{LMP10} that the FOI $K^{(l)}$ at $(\x)^l$ is associated with the pair
\[
(F^{(l)}, \mu^{\otimes l})
\]
on $U^l$, where $\mu$ is a trace density of the star product~$\star$.

 The main goal of this paper is to develop an algebraic framework which will allow to express the jet of the formal oscillatory exponent $\exp G^{(l)}$ at $(\x)^l \in M^l$ for every $l \geq 1$ in terms of the star product~$\star$. We plan to apply this framework to the Berezin and Berezin-Toeplitz star products on general symplectic manifolds because we expect that there should exist an analog of the formal Calabi function $G^{(l)}$ for these star products, which may shed some light on their inner structure.

 \section{The algebra $\mathbb{B}$}\label{S:algb}

Given a manifold $M$, we identify the diagonal of $M^l$ with $M$ (thus assuming that $M \subset M^l$ for any $l$). An $l$-differential operator $C(f_1, \ldots, f_l)$ on $M$ defines a mapping
\[
C: C^\infty(M^l,M) \to C^\infty(M).
\]

Let  $ \mathbb{A}:=(C^\infty(M)[[\nu]],\star)$ be a star algebra on a Poisson manifold~$M$ with natural star product $\star$. 
Denote by $\widetilde M$ a copy of $M$ with the opposite Poisson structure. The opposite product $f \star^{\mathrm{opp}} g := g \star f$ is a star product on $\widetilde M$. The product
\[
     \odot := \star \otimes \star^{\mathrm{opp}}
\]
is a natural star product on $M \times \widetilde{M}$. For $f,g,u,v \in \mathbb{A}$ we have
\[
    (f \otimes g) \odot (u \otimes v) = (f \star u) \otimes (v \star g).
\]
Here $\otimes$ is the tensor product over the ring $\C[[\nu]]$. The product $\odot$ induces a product on $C^\infty(M \times M,M)[[\nu]]$ which will be denoted by the same symbol. We introduce an algebra
\[
\mathbb{B}:= (C^\infty(M \times M, M)[[\nu]], \odot).
\]
We call an element of $\mathbb{B}$ factorizable if it is induced by a formal function $f \otimes g \in \mathbb{A} \otimes \mathbb{A}$, and use the same notation $f \otimes g$ for this element. Recall that $\mathfrak{N}$ denotes the algebra of natural operators on $M$.
There exists a homomorphism $F \mapsto N_F$ from $\mathbb{B}$ to $\mathfrak{N}$ given on the factorizable elements by
\[
            N_{f \otimes g} = L_fR_g.
\]
We will prove that if $M$ is symplectic, then this mapping is an isomorphism. To this end, we need to recall several definitions and facts from~\cite{LMP5}. 

If $A$ is a differential operator of order $r$ on a manifold $M$, then its principal symbol $\mathrm{Symb}_r(A)$ is a fiberwise polynomial function  of degree~$r$ on the cotangent bundle $T^\ast M$. The mapping $A \mapsto \mathrm{Symb}_r(A)$ extends by zero to the differential operators of order less than $r$. Given a natural operator $N = N_0 + \nu N_1 + \ldots$ on $M$, we call the formal series
\[
   \sigma(N) := \sum_{r=0}^\infty \mathrm{Symb}_r(N_r)
\]
the sigma symbol of $N$. It can be interpreted as a function on the formal neighborhood of the zero section $Z$ of $T^\ast M$,
\[
     \sigma(N) \in C^\infty(T^\ast M,Z).
\]
The mapping $N \mapsto \sigma(N)$ is a surjective homomorphism from $\mathfrak{N}$ onto $C^\infty(T^\ast M,Z)$ whose kernel is $\nu\mathfrak{N}$. It follows that the sigma symbol $\sigma(N_F)$ of $F = F_0 + \nu F_1 + \ldots \in C^\infty(M\times M, M)[[\nu]]$ depends only on~$F_0$.
It was proved in \cite{LMP5} that if $M$ is symplectic, then the mapping
\[
      C^\infty(M \times M, M) \ni F_0   \mapsto \sigma(N_{F_0}) 
\]
is an isomorphism of $C^\infty(M \times M, M)$ onto $C^\infty(T^\ast M,Z)$. 
\begin{theorem}\label{T:isom}
If $\star$ is a natural star product on a symplectic manifold~$M$, then the mapping $F \mapsto N_F$ is an isomorphism of the algebra $\mathbb{B}$ onto $\mathfrak{N}$. 
\end{theorem}
\begin{proof}
We will construct the inverse mapping of the mapping $F \mapsto N_F$. Let $N$ be an arbitrary natural operator on $M$. There exists a unique element $F_0 \in C^\infty(M \times M, M)$ such that
\[
      \sigma(N_{F_0}) = \sigma(N).     
\]
Then $\nu^{-1}(N - N_{F_0}) \in \mathfrak{N}$. Let $F_1$ denote the unique element of $C^\infty(M \times M, M)$ such that
\[
       \sigma(N_{F_1}) = \sigma(\nu^{-1}(N - N_{F_0})).
\]
Hence, $\nu^{-2}(N - N_{F_0} - \nu N_{F_1}) \in \mathfrak{N}$. Continuing this process, we pro\-duce a unique element $F = F_0 + \nu F_1 + \ldots \in\mathbb{B}$ such that~$N = N_F$.
\end{proof}

We fix a point $\x\in M$ and denote by~$\Ncal$ the set of all natural distributions supported at $\x$. 

Denote by $\tau$ the involution on $\mathbb{B}$ such that $\tau(f\otimes g) = g\otimes f$. It is an antiautomorphism of 
$\mathbb{B}$. The algebra $\mathbb{B}$ acts on $\Ncal$ so that an element $F\in \mathbb{B}$ maps $\Lambda\in\Ncal$ to $\Lambda \circ N_{\tau(F)}\in \Ncal$. Given $F \in \mathbb{B}$ and $\x \in M$, we set
\begin{equation}\label{E:LF}
   \Lambda_F := \delta_\x \circ N_{\tau(F)}.
\end{equation}
For a factorizable $F=f \otimes g$ we have
\begin{eqnarray*}
\Lambda_{f \otimes g} (h)= (N_{\tau(f \otimes g)}h)(\x) =
(N_{g \otimes f}h)(\x) =\\
 (L_gR_fh)(\x) =(g \star h \star f)(\x).
\end{eqnarray*}
In the next section we construct one more auxiliary algebra based on the algebra $\mathbb{B}$.

\section{The algebra $\calC$}\label{S:algc}

Let $\star$ be a star product of the anti-Wick type on a pseudo-K\"ahler manifold $M$, $B$ be its formal Berezin transform,  and $\x$ be a fixed point in $M$.  We choose a coordinate chart $U$ containing $\x$ with coordinates $\{z^k,\bar z^l\}$ such that $z^k(\x)=\bar z^l(\x)=0$ for all $k,l$. We consider various jet spaces on $M$ at the point $\x$ and on $M \times M$ at the point $(\x,\x)$. These spaces will be identified with spaces of formal series in local coordinates. Denote by $\F=\C[[\nu,z,\bar z]]$ the space of $\nu$-formal jets on $M$ at $\x$ and by $\A=(\F,\star)$ the algebra on $\F$ with the product induced by $\star$. Denote by $\F^{(2)} = \C[[\nu, z, \bar z, w, \bar w]]$ the space of $\nu$-formal jets on $M \times M$ at $(\x,\x)$, where $\{z^k,\bar z^l\}$ and $\{w^k,\bar w^l\}$ are the coordinates on the first and the second factors of the chart $U \times U$, respectively. For the involutive mapping $\tau: \F^{(2)} \to \F^{(2)}$ such that $\tau(f \otimes g) = g \otimes f$ for $f,g \in \F$, one has $\tau(z^k)=w^k$ and $\tau(\bar z^l)=\bar w^l$. Since the natural distribution (\ref{E:LF}) supported at $\x$ depends only on the jet of $F$ at $(x,x)$, the mapping $F \to \Lambda_F$ induces a mapping $\lambda: \F^{(2)} \to \Ncal$ such that
\begin{equation}\label{E:factx}
\lambda(f \otimes g)= (g \star h \star f)(\x).
\end{equation}
The space $\J =\C[[z,\bar z]]$ of jets on $M$ at $\x$ has a descending filtration $\J=F_0\J \supset F_1\J \supset \ldots$,
where $F_r\J$ is the space of jets which have zero of order at least $r$ at~$\x$. We assume that $F_r\J=\J$ for $r <0$. We introduce a filtration
\[
\F=F_0\F \supset F_1\F \supset \ldots
\]
on the space of formal jets $\F=\J[[\nu]]$ which agrees with the filtration on $\J$ and for which the filtration degree of $\nu$ is 2, so that
\[
    F_r\F = F_r \J + \nu F_{r-2} \J  + \nu^2 F_{r-4} \J+ \ldots.
\]
We call it {\it the standard filtration}. Observe that $\F/F_r \F$ is a finite dimensional vector space over $\C$. One can check that
\[
      \F^{(2)} = \varprojlim_r (\F \otimes \F)/F_r(\F \otimes \F),
\]
where the subspaces
\[
     F_r(\F \otimes \F) := \sum_{i+j=r} F_i\F \otimes F_j\F
\]
form the standard filtration on $\F \otimes \F$. Here $\otimes$  is the tensor product over the ring $\C[[\nu]]$. 
\begin{lemma}\label{L:filtered}
  The algebra $\A=(\F,\star)$ is a filtered algebra with respect to the standard filtration.
\end{lemma}
\begin{proof}
Since $\star$ is a natural star product, the bidifferential operator $C_r$ in (\ref{E:star}) is of order not greater than $r$ in each argument. Therefore, if $f \in F_i\J$ and $g \in F_j\J$, then $C_r(f,g) \in F_{i+j - 2r}\J$ and $\nu^r C_r(f,g)\in F_{i+j}\F$, 
whence the lemma follows.
\end{proof}
Lemma \ref{L:filtered} allows to extend various mappings of the space $\F \otimes \F$ related to the product $\star$ to its completion $\F^{(2)} = \F \hat \otimes \F$ with respect to the topology associated with the standard filtration. We will tacitly assume that these extensions can be justified with the use of this lemma.

We define a filtered associative algebra $\calC:=(\F^{(2)}, \ast)$ with the product $\ast$ given on the factorizable elements by the formula
\begin{eqnarray*}
    (g_1 \otimes h_1) \ast (g_2 \otimes h_2):=(h_1 \star g_2)(\x) \cdot (g_1 \otimes h_2).
\end{eqnarray*}

We introduce a trace on $\calC$
given on the factorizable elements by the formula
\begin{equation}\label{E:trace}
    \tr (f \otimes g) := (g \star f) (\x).
\end{equation}
One can check the trace property on factorizable elements,
\begin{eqnarray*}
  \tr((g_1 \otimes h_1) \ast (g_2 \otimes h_2)) =(h_1 \star g_2)(\x) \cdot (h_2 \otimes g_1)(\x) =\\
   \tr((g_2 \otimes h_2) \ast (g_1 \otimes h_1)).
\end{eqnarray*}

Given a factorizable element $f \otimes g \in \F^{(2)}$, we get from formula~(\ref{E:factx})  that
\[
  \tr  (f \otimes g)  = (g \star f)(\x) = \langle\lambda (f \otimes g), 1\rangle.
\]
It follows that for $F\in \calC$,
\begin{equation}\label{E:lambdatrace}
  \tr  F  =  \langle\lambda (F), 1\rangle.
\end{equation}

We introduce a splitting of $\calC$,
\begin{equation}\label{E:split}
     \calC = \G \oplus \H,
\end{equation}
where $\G= \C[[\nu, z, \bar w]]$ and $\H$ is generated by $\bar z^l$ and $w^k$ for all $k,l$, i.e., any $H\in \H$ can be represented as
\[
H= \bar z^l A_l + w^k B_k 
\]
for some $A_l,B_k \in \calC$. This splitting does not depend on the choice of local holomorphic coordinates used in its definition.
We will show that in the splitting (\ref{E:split}) the subspace $\G$ is a subalgebra of $\calC$ and $\H$ is a two-sided ideal of $\calC$.
\begin{lemma}\label{L:lsurj}
  The subspace $\H\subset \calC$ is a two-sided ideal of the algebra~$\calC$ which lies in the kernel of the mapping $\lambda$.
  \end{lemma}
\begin{proof}
It suffices to check the statement of the lemma on the generators
\[
  U^l :=(\bar z^l u) \otimes v = (u \star \bar z^l) \otimes v \mbox{ and } V^k := u \otimes (z^kv) = u \otimes (z^k \star v)
\]
of $\H$ and factorizable $F = f \otimes g \in \calC$, where $u,v,f,g \in \F$ are arbitrary. We have
 \begin{eqnarray*}
 F \ast U^l = (f \otimes g) \ast ((u \star \bar z^l) \otimes v) =\hskip 4cm\\
  (g \star u \star \bar z^l)(\x) \cdot (f \otimes v)= 
  ((g \star u)  \bar z^l)(\x) \cdot (f \otimes v)=0,
 \end{eqnarray*}
 because $\bar z^l(\x)=0$. Then we see that
 \[
 U^l \ast F = ((\bar z^l u) \otimes v) \ast (f \otimes g) = (v \star f)(\x) \cdot ((\bar z^l u) \otimes g) \in \H.
 \]
 One can check similarly that $F \star V_k \in \H$ and $V_k \star F =0$. It follows that $\H$ is a two-sided ideal of $\calC$.
 We get from formula (\ref{E:factx}) that for any $h \in \F$,
 \[
   \langle \lambda(U^l), h \rangle = (v \star h \star u \star \bar z^l)(\x) = ((v \star h \star u)\bar z^l)(\x) =0,
 \]
 because $\bar z^l(\x)=0$. Thus, $\lambda(U^l)=0$. One can similarly check that $\lambda(V^k)=0$,
 which implies the second statement of the lemma.
\end{proof}

\begin{lemma}\label{L:isomalg}
The subspace $\G\subset \calC$ is a subalgebra of $\calC$ isomorphic to the algebra $\calC/\H$.
\end{lemma}
\begin{proof}
The space $\G$ is topologically generated by the elements $a \otimes b$, where $a \in \C[[\nu, z]]$ is formally holomorphic and $b \in \C[[\nu,\bar w]]$ is antiholomorphic.
We have
\[
    (a_1 \otimes b_1) \ast  (a_2 \otimes b_2) = (b_1 \star a_2)(\x)\cdot (a_1 \otimes b_2)  \in \G.
 \]
 Therefore, $\G$ is a subalgebra of $\calC$. Clearly, it is isomorphic to the algebra $\calC/\H$.
\end{proof}

Let $\alpha: C^\infty(U)[[\nu]] \to \F$ be the mapping that maps $f$ to its jet at~$\x$. It is surjective by Borel's lemma. We define a mapping 
\[
\gamma: C^\infty(U)[[\nu]]  \to \G
\]
as follows. Given $f \in C^\infty(U)[[\nu]]$, let $\tilde f \in C^\infty(U \times \widebar U)[[\nu]]$ be an almost analytic extension of $f$. We set $\gamma(f)$ equal to the jet  of $\tilde f$ at $(\x,\x)$. This jet lies in $\G$ and does not depend on the choice of the almost analytic extension of $f$. The mapping $\gamma$ is surjective. There is a bijection $\beta: \F \to \G$ such that $\gamma=\beta \circ \alpha$. In coordinates, 
\[
\beta: f(z,\bar z) \mapsto f(z, \bar w).
\]

\begin{lemma}\label{L:lgamma}
Given $g \in C^\infty(U)[[\nu]]$, the following formula holds,
\[
    \lambda(\gamma(g)) = \delta_\x \circ B \circ g \circ B^{-1}.
\]
\end{lemma}
\begin{proof}
Let $g = ab$, where $a$ is a holomorphic and $b$ is an antiholomorphic function on $U$. Then $\gamma(g)= a \otimes b$. Using formula (\ref{E:iaib}), we get that
 \begin{eqnarray*}
     \lambda(\gamma(g)) = \delta_\x \circ N_{\tau(a \otimes b)}= \delta_\x \circ N_{b \otimes a} =
      \delta_\x \circ (R_aL_b) =\\
       \delta_\x \circ B \circ (ab) \circ B^{-1} = \delta_\x \circ B \circ g \circ B^{-1}.
 \end{eqnarray*}
 For a generic $g \in C^\infty(U)[[\nu]]$, the distribution $\delta_\x \circ B \circ g \circ B^{-1}$ depends only on the jet of $g$ at $\x$ and the space $\F$ is topologically generated by the elements $\alpha(ab)$. Therefore, the lemma follows from the calculation above.
\end{proof}
\begin{lemma}\label{L:linj}
The restriction of the mapping $\lambda$ to $\G$, $\lambda|_\G : \G \to \Ncal$, is injective.
\end{lemma}
\begin{proof}
  Let $G$ be an arbitrary element of $\G$ which lies in the kernel of~$\lambda$. There exists $g \in C^\infty(U)[[\nu]]$ such that $G = \gamma(g)$. Then for any $h \in C^\infty(U)[[\nu]]$ we have from Lemma \ref{L:lgamma} that
 \begin{equation}\label{E:izero}
      B(g \cdot B^{-1}h)(\x)= \langle \lambda(\gamma(g)), h \rangle = \langle \lambda(G),h\rangle = 0.
  \end{equation}
  It was proved in \cite{KS} that the distribution $f \mapsto (Bf)(\x)$ is a FOI at $\x$. By Lemma \ref{L:nondeg}, the pairing $u, v \mapsto B(u \cdot v)(\x)$ on $C^\infty(U)[[\nu]]$ induces a nondegenerate pairing on $\F$. Since $B^{-1}h$ is an arbitrary element of $C^\infty(U)[[\nu]]$, we see from (\ref{E:izero}) that the jet of $g$ at $\x$ is zero. Therefore, $G=0$, whence the lemma follows.
\end{proof}
\begin{corollary}\label{C:restl}
  The ideal $\H$ is the kernel of the mapping $\lambda$ and the mapping $\lambda|_\G : \G \to \Ncal$ is bijective.
\end{corollary}
\begin{proof}
The mapping $\lambda$ is surjective. It was proved in Lemma \ref{L:lsurj} that $\H$ lies in the kernel of $\lambda$. The corollary follows from the splitting (\ref{E:split}) and Lemma~\ref{L:linj}.
\end{proof}
\section{The algebra of distributions}\label{S:algd}

Corollary \ref{C:restl} implies that one can transfer the product $\ast$ from the algebra $\calC$ to $\Ncal$. We denote the resulting product on $\Ncal$ by $\bullet$. It follows from (\ref{E:lambdatrace}) and Lemma \ref{L:isomalg} that
the algebra $(\Ncal,\bullet)$ is isomorphic to the algebra $(\G,\ast) \cong \calC/\H$. The mapping
\begin{equation}\label{E:trn}
\Ncal \ni u \mapsto \langle u,1\rangle
\end{equation}
is a trace on the algebra $(\Ncal, \bullet)$. Its pullback via the mapping $\lambda$ is the trace $\tr$ on $\calC$.

In the rest of the paper we will express the trace of the product of $l$ elements of the algebra $(\Ncal,\bullet)$ in terms of the formal $l$-point Calabi function of the star product $\star$.

A differential operator on $M$ induces an operator on $\F$ which we call a differential operator on $\F$. The standard filtration on $\F$ induces a filtration on the operators on $\F$, which we also call standard.  Namely, the filtration degree of an operator $A$ on $\F$ is the largest integer $k$ such that $A(F_r\F) \subset F_{r+k}\F$ for all $r$ (we assume that $F_r\F=\F$ for $r <0$). If $A$ is a differential operator of order $r$ which does not depend on $\nu$, its filtration degree is at least $-r$. We denote by~$\mathfrak{N}_\x$ the algebra of natural operators on $\F$. These operators are induced by the operators from $\mathfrak{N}$. Observe that if $N = N_0 + \nu N_1 + \ldots$ is a natural operator, then the filtration degree of $\nu^r N_r$ is at least $r$.

In the remainder of this section $\varphi= \nu^{-1}\varphi_{-1}+\varphi_0 + \ldots$ is a formal function on $M$ such that $\x$ is a critical point of $\varphi_{-1}$ with zero critical value, $\varphi_{-1}(\x)=0$. We do not assume that the critical point $\x$ is nondegenerate. Observe that the filtration degree of $\varphi$ is a least zero. Recall that the adjoint action of an operator $A$ on an operator $B$ is denoted by $\ad(A):B \mapsto [A,B]$.
\begin{lemma}\label{L:ene}
If $N \in \mathfrak{N}_\x$, then $e^{-\varphi} N e^\varphi \in \mathfrak{N}_\x$.
\end{lemma}
\begin{proof}
Assume that $N = N_0 + \nu N_1 + \ldots \in \mathfrak{N}_\x$. Then for each $r\geq 0$ the formal differential operator
\begin{equation}\label{E:ene}
    e^{-\varphi} (\nu^r N_r) e^\varphi = \sum_{k=0}^r \frac{1}{k!} (-\ad \varphi)^k (\nu^r N_r)
\end{equation}
is of order not greater than $r$. The operator (\ref{E:ene}) is natural and its $\nu$-filtration degree is at least zero. Its standard filtration degree is at least $r$. Therefore, the series
\[
     e^{-\varphi}Ne^\varphi = \sum_{r=0}^\infty e^{-\varphi} (\nu^r N_r) e^\varphi 
\]
converges to an element of $\mathfrak{N}_\x$ in the topology associated to the standard filtration.
\end{proof}

Below we define an action $e^\varphi: u \mapsto u \circ e^\varphi$ on $\Ncal$ which behaves like a composition.  However, the multiplication operator by the formal oscillatory exponent $e^\varphi$ is not a natural operator, because the Taylor series of $e^\varphi$ at $\x$ contains negative powers of $\nu$. 

Given $u \in \Ncal$, there exists $N \in \mathfrak{N}_\x$ such that $u = \delta_\x \circ N$. We set
\[
    u \circ e^\varphi := e^{\varphi(\x)} \delta_\x \circ (e^{-\varphi} N e^\varphi).
\]
Since $\varphi_{-1}(\x)=0$, we see that $e^{\varphi(\x)} \in \C[[\nu]]$. By Lemma \ref{L:ene}, $u \circ e^\varphi$ is an element of $\Ncal$. We will show that it does not depend on the choice of $N$.
\begin{lemma}
   If $u$ has two different representations $u = \delta_\x \circ N = \delta_\x \circ {\tilde N}$ for $N, \tilde N \in \mathfrak{N}_\x$,
   then $e^{\varphi(\x)} \delta_\x \circ (e^{-\varphi} N e^\varphi) = e^{\varphi(\x)} \delta_\x \circ (e^{-\varphi} \tilde N e^\varphi)$.
\end{lemma}
\begin{proof}
We have $\delta_\x \circ (N - \tilde N)=0$. Therefore, in coordinates, one can write $N - \tilde N = z^k A_k + \bar z^l B_l$ for some $A_k, B_l  \in \mathfrak{N}_\x$. We need to show that
\[
    e^{\varphi(\x)} \delta_\x \circ (e^{-\varphi} (N - \tilde N) e^\varphi)=0,
\]
which follows from the observation that
\[
   e^{-\varphi} (N - \tilde N) e^\varphi = z^k e^{-\varphi} A_k e^\varphi + \bar z^l e^{-\varphi} B_l e^\varphi
\]
and the fact that $z^k(\x)=\bar z^l(\x)=0$ for all $k,l$.
\end{proof}
\begin{lemma}\label{L:compe}
Let $\varphi= \nu^{-1}\varphi_{-1}+\varphi_0 + \ldots$ and $\psi= \nu^{-1} \psi_{-1}+\psi_0 + \ldots$ be formal functions on $M$ such that $\x$ is a critical point of $\varphi_{-1}$ and $\psi_{-1}$ with zero critical value, $\varphi_{-1}(\x) = \psi_{-1}(\x)=0$. Then for any $u \in \Ncal$ one has
\[
     (u \circ e^\varphi) \circ e^\psi = u \circ e^{\varphi+\psi}.
\]
\end{lemma}
\begin{proof}
Let $N \in \mathfrak{N}_\x$ be such that $u = \delta_\x \circ N$. Then
\begin{eqnarray*}
    (u \circ e^\varphi) \circ e^\psi  = (e^{\varphi(\x)} \delta_\x \circ (e^{-\varphi} N e^\varphi)) \circ e^\psi =\\
    e^{\varphi(\x)+\psi(\x)} \delta_\x \circ (e^{-\psi} e^{-\varphi} N e^\varphi e^\psi) = u \circ e^{\varphi+\psi}.
\end{eqnarray*}
\end{proof}
We introduce a $\nu$-linear functional $K:\Ncal \to \C[[\nu]]$,
\[
      K(u) := \langle u \circ e^\varphi, 1 \rangle.
\]
If $A$ is a differential operator on $M$, we denote by $A^t$ its transpose that acts on a distribution $u$ as $A^t u := u \circ A$. Let $v$ be a vector field on $M$. Since $\nu v$ and $\nu v\varphi$ are natural operators, then for $u \in \Ncal$ we get that $(\nu v -\nu v \varphi)^tu \in \Ncal$.

\begin{lemma}\label{L:vt}
For any $u \in \Ncal$, $((\nu v -\nu v \varphi)^tu) \circ e^\varphi= (u \circ e^\varphi) \circ (\nu v)$.
\end{lemma}
\begin{proof}
Assume that $u = \delta_\x \circ N$ for some $N \in \mathfrak{N}_\x$. Then
\[
    (\nu v -\nu v \varphi)^tu = \delta_\x \circ N \circ (\nu v - \nu v \varphi).
\]
Therefore,
\begin{eqnarray*}
((\nu v -\nu v \varphi)^tu) \circ e^\varphi = e^{\varphi(\x)} \delta_\x \circ (e^{-\varphi} (N \circ (\nu v - \nu v \varphi)) e^\varphi)=\\
e^{\varphi(\x)} \delta_\x \circ (e^{-\varphi} N e^\varphi) \circ (\nu v) = (u \circ e^\varphi) \circ (\nu v),
\end{eqnarray*}
because $e^{-\varphi} \circ (v -  v \varphi) \circ e^\varphi = v$.
\end{proof}
\begin{corollary}
For any $u \in \Ncal$, $K((\nu v -\nu v \varphi)^tu)=0$.
\end{corollary}
\begin{proof}
We have by Lemma \ref{L:vt} that
\begin{eqnarray*}
   K((\nu v -\nu v \varphi)^tu) = \langle ((\nu v -\nu v \varphi)^tu) \circ e^\varphi, 1 \rangle = \\
   \langle (u \circ e^\varphi) \circ (\nu v), 1 \rangle = \langle u \circ e^\varphi, (\nu v) 1 \rangle =0.
\end{eqnarray*}
\end{proof}
\begin{theorem}\label{T:ident}
Let $S :\Ncal \to \C[[\nu]]$ be a $\nu$-linear functional such that the equality
\[
   S\left((\nu v - \nu v\varphi)^tu\right)=0
\]
holds for any vector field $v$ and any $u \in \Ncal$. Then there exists a formal constant $c(\nu) \in \C[[\nu]]$ such that
\[
     S(u) = c(\nu)\langle u \circ e^\varphi, 1 \rangle.
\]
\end{theorem}
\begin{proof}
Consider a functional $T: \Ncal \to \C[[\nu]]$ given by the formula
\[
     T(u) := S(u \circ e^{-\varphi}).
\]
We will show that $T\left((\nu v)^tu\right)=0$ for any vector field $v$ and any $u \in \Ncal$. Let $N \in \mathfrak{N}_\x$ be such that $u = \delta_\x \circ N$. Given a vector field $v$ and $u \in \Ncal$, we have
\begin{eqnarray*}
  T\left((\nu v)^tu\right) = S(((\nu v)^t u) \circ e^{-\varphi}) = S((\delta_\x \circ (N \circ (\nu v)) \circ e^{-\varphi})=\\
  S(e^{-\varphi(\x)}\delta_\x \circ (e^{\varphi} (N \circ (\nu v)) e^{-\varphi} )) = \hskip 3.8cm\\
  S(e^{-\varphi(\x)}\delta_\x \circ (e^{\varphi} N e^{-\varphi}  \circ (\nu v - \nu v \varphi)))=\hskip 3.1cm\\
  S((u \circ e^{-\varphi}) \circ ((\nu v - \nu v \varphi))) = S((\nu v - \nu v \varphi)^t (u \circ e^{-\varphi})) =0.
\end{eqnarray*}
In local coordinates one can write any operator $N \in \mathfrak{N}_\x$ as
\[
    N = f + A^p \circ \left(\nu \frac{\p}{\p z^p}\right) + B^q \circ \left(\nu \frac{\p}{\p \bar z^q}\right),
\]
where $f=N1 \in \C[[\nu,z, \bar z]]$ and $A^p, B^q \in \mathfrak{N}_\x$. Then for $u = \delta_\x \circ N$ we have
\begin{eqnarray*}
     T(u) = T\left(\delta_\x \circ \left(f + A^p \circ \left(\nu \frac{\p}{\p z^p}\right) + B^q \circ \left(\nu \frac{\p}{\p \bar z^q}\right)
     \right)\right)=\\
    f(\x) T(\delta_\x) + T\left(\left(\nu \frac{\p}{\p z^p}\right)^t (\delta_\x \circ A^p)\right) + \hskip 1cm\\
     T\left(\left(\nu \frac{\p}{\p \bar z^q}\right)^t (\delta_\x \circ B^q)\right) = f(\x) T(\delta_\x).
\end{eqnarray*}
It follows that $T(u) = T(\delta_\x) \langle u, 1 \rangle$. We set $c(\nu) := T(\delta_\x)$. Using Lemma~\ref{L:compe}, we get that
\[
    S(u) = T(u \circ e^\varphi) = c(\nu) \langle u \circ e^\varphi, 1 \rangle.
\]
\end{proof}

Let $\x$ be a point in $M$, $U$ be a contractible coordinate chart with coordinates $\{z^p, \bar z^q\}$ such that $z^p(\x)=\bar z^q(\x)=0$ for all $p,q$, and $\Phi$ be a potential of the classifying form $\omega$ of the star product $\star$ on $U$. We choose an almost analytic extension $\tilde\Phi$ of $\Phi$ on $U \times \widebar{U}$.
In Section \ref{S:sep} we introduced the cyclic function
\[
     G^{(l)}(x_1,\ldots, x_l) = \tilde\Phi(x_1,x_2) + \ldots  + \tilde \Phi(x_l,x_1)   
     - (\Phi(x_1) + \ldots +  \Phi(x_l))
\]
on the neighborhood $U^l$ of the diagonal point $(\x)^l$ of $M^l$. The jet of the function $G^{(l)}$ at $(\x)^l \in M^l$ is given in local coordinates by the formula
\begin{eqnarray*}
G^{(l)}(z, \bar z) = \Phi(z_1, \bar z_2) + \Phi(z_2, \bar z_3) + \ldots + \Phi(z_l, \bar z_1)\hskip 2cm\\
- (\Phi(z_1,\bar z_1) + \ldots + \Phi(z_l, \bar z_l)) \in \nu^{-1}\C[[\nu, z_1, \bar z_1, \ldots z_l, \bar z_l]],
\end{eqnarray*}
where we have used the notations  $z = (z_1, \ldots, z_l), z_i = (z_i^1, \ldots, z_i^m), \bar z = (\bar z_1, \ldots, \bar z_l), \bar z_j = (\bar z_j^1, \ldots, \bar z_j^m),$ and $m = \dim_{\C} M$. This is the jet of the formal $l$-point Calabi function of $\omega$ at $(\x)^l$.
\begin{lemma}\label{L:diagp}
The diagonal point $(\x)^l \in M^l$ is a critical point of the function $G^{(l)}$ with zero critical value.
\end{lemma}
\begin{proof}
Clearly, $G^{(l)}((\x)^l)=0$. In local coordinates,
\begin{eqnarray}\label{E:diagp}
    \frac{\p G^{(l)}}{\p z_i^p} = \frac{\p \Phi}{\p z^p}(z_i, \bar z_{i+1}) - \frac{\p \Phi}{\p z^p}(z_i, \bar z_i)   \mbox{ and } \\
    \frac{\p G^{(l)}}{\p \bar z_j^q} = \frac{\p \Phi}{\p \bar z^q}(z_{j-1}, \bar z_j) - \frac{\p \Phi}{\p \bar z^q}(z_j, \bar z_j), \nonumber 
\end{eqnarray}
where we identify $\bar z_{l+1}$ with $\bar z_1$ and $z_0$ with $z_l$. Therefore,
\[
    \frac{\p G^{(l)}}{\p z_i^p}((\x)^l) = 0 \mbox{ and } \frac{\p G^{(l)}}{\p \bar z_j^q}((\x)^l) =0.
\]
\end{proof}
\noindent {\it Remark.} The point $(\x)^l \in M^l$ is a degenerate critical point of the function $G^{(l)}$, but it is a nondegenerate critical point of the function~(\ref{E:fgl}).
\begin{theorem}\label{T:main}
 The following identity holds for any natural distributions $u_1, \ldots, u_m \in \Ncal$, 
 \begin{equation}\label{E:main}
     \langle u_1 \bullet \ldots \bullet u_l , 1 \rangle = \langle (u_1 \otimes \ldots \otimes u_l)\circ \exp G^{(l)}, 1 \rangle. 
 \end{equation}
\end{theorem}
\noindent {\it Remark.} Observe that the left-hand side of (\ref{E:main}) is the trace of the product $u_1 \bullet \ldots \bullet u_l$ in the algebra  $(\Ncal, \bullet)$, which agrees with the fact that $G^{(l)}$ is cyclic. The distribution $u_1 \otimes \ldots \otimes u_l$ on the right-hand side is a natural distribution on $M^l$ supported at $(\x)^l$. Since the natural distributions $u_1, \ldots, u_l$ are arbitrary, the element $\exp G^{(l)}$ is completely (though not explicitly) determined in terms of the star product $\star$.
\begin{proof}
We introduce a functional $W^{(l)}$ on the space of natural distributions on $M^l$ supported at the point $(\x)^l$ by the formula
\[
     W^{(l)}(u_1 \otimes \ldots \otimes u_l):= \langle u_1 \bullet \ldots \bullet u_l , 1 \rangle. 
\]
Suppose that $u_i = \lambda(f_i \otimes g_i)$ for $1 \leq i \leq l$, where $f_i,g_i \in \F$ are arbitrary. Then, by formula (\ref{E:factx}),
\begin{eqnarray*}
   W^{(l)}(u_1 \otimes \ldots \otimes u_l) =  (g_1 \star f_2)(\x) \cdot (g_2 \star f_3)(\x)\cdot  \ldots (g_l \star f_1)(\x).
\end{eqnarray*}
Observe that $\delta_\x = \lambda(1 \otimes 1)$ and $\delta_\x \bullet \delta_\x=\delta_\x$.  Clearly, 
\[
W^{(l)}\left (\delta_{(\x)^l}\right )=1 \mbox{ and }  \langle \delta_{(\x)^l}\circ \exp G^{(l)}, 1 \rangle =  \langle \delta_{(\x)^l}, 1 \rangle = 1, 
\]
where we have used that $\delta_{(\x)^l} = \delta_\x \otimes \ldots \otimes \delta_\x$ and $G^{(l)}((\x)^l)=0$. According to Theorem~\ref{T:ident}, in order to prove formula (\ref{E:main}) it remains to verify that for any $i,j,p,q$,
\begin{equation}\label{E:annih}
    W^{(l)} \circ \left(\nu \frac{\p}{\p z_i^p} - \nu \frac{\p G^{(l)}}{\p z_i^p} \right)^t = 0 \mbox{ and } W^{(l)} \circ \left(\nu \frac{\p}{\p \bar z_j^q}- \nu \frac{\p G^{(l)}}{\p \bar z_j^q}\right)^t = 0.
\end{equation}
We will check the first equality on the elements $u_1 \otimes \ldots \otimes u_l$ with $u_i = \lambda(f_i \otimes g_i)$, which topologically generate $\Ncal$. We use formula (\ref{E:diagp}) to calculate the action of 
\[
    \left( \nu \frac{\p}{\p z_i^p}- \nu \frac{\p G^{(l)}}{\p z_i^p} \right)^t = \left( \nu \frac{\p}{\p z_i^p} - \nu \frac{\p \Phi}{\p z^p}(z_i, \bar z_{i+1}) + \nu \frac{\p \Phi}{\p z^p}(z_i, \bar z_i)  \right)^t 
\]
on $u_1 \otimes \ldots \otimes u_l$. The operator
\[
    \left(\nu \frac{\p}{\p z_i^p} + \nu \frac{\p \Phi}{\p z^p}(z_i, \bar z_i)\right)^t
\]
acts only on the factor $u_i$ in $u_1 \otimes \ldots \otimes u_l$. We have
\begin{eqnarray*}
   \left\langle \left(\nu \frac{\p}{\p z^p} + \nu \frac{\p \Phi}{\p z^p}\right)^t u_i, h\right\rangle =
    \left\langle u_i, \nu\frac{\p \Phi}{\p z^p} \star h \right\rangle =\hskip 2cm\\
     \left(g_i \star \nu\frac{\p \Phi}{\p z^p} \star h \star f_i\right)(\x) =\left\langle \lambda\left(f_i \otimes \left(g_i \star \nu\frac{\p \Phi}{\p z^p}\right)\right), h \right\rangle.
\end{eqnarray*}
We introduce an element
\[
   \hat u_i := \left(\nu \frac{\p}{\p z^p}+ \nu \frac{\p \Phi}{\p z^p}\right)^t u_i = \lambda\left(f_i \otimes \left(g_i \star \nu\frac{\p \Phi}{\p z^p}\right)\right).
\]
We get
\begin{eqnarray*}
   W^{(l)}(u_1 \otimes \ldots \otimes \hat u_i \otimes \ldots \otimes u_l) =
   \langle u_1 \bullet \ldots \bullet \hat u_i \bullet \ldots \bullet u_l, 1 \rangle=\\
    (g_1 \star f_2)(\x) \cdot (g_2 \star f_3)(\x)\cdot  \ldots \cdot (g_{i-1} \star f_i)(\x) \cdot\hskip 2cm\\
\left(g_i \star \nu\frac{\p \Phi}{\p z^p} \star f_{i+1}\right)(\x) \cdot (g_{i+1} \star f_{i+2})(\x) \cdot\ \ldots \cdot (g_l \star f_1)(\x).
\end{eqnarray*}
It remains to calculate
\[
    W^{(l)} \left(\left(\nu\frac{\p \Phi}{\p z^p}(z_i, \bar z_{i+1})\right)^t (u_1 \otimes \ldots \otimes u_l) \right).
\]
The jet $\nu\frac{\p \Phi}{\p z^p}(z_i, \bar z_{i+1})$ can be expressed as the following series convergent in the topology associated to the standard filtration,
\[
    \nu\frac{\p \Phi}{\p z^p}(z_i, \bar z_{i+1}) = \sum_\alpha a_\alpha(z_i) b_\alpha(\bar z_{i+1}).
\]
We have
\begin{eqnarray*}
\left\langle \left(\nu\frac{\p \Phi}{\p z^p}(z_i, \bar z_{i+1})\right)^t (u_1 \otimes \ldots \otimes u_l),  h_1 \otimes \ldots \otimes h_l\right\rangle =\hskip 1cm\\
\left\langle u_1 \otimes \ldots \otimes u_l, \nu\frac{\p \Phi}{\p z^p}(z_i, \bar z_{i+1}) (h_1 \otimes \ldots \otimes h_l)\right\rangle =\hskip 1cm\\
\left\langle u_1 \otimes \ldots \otimes u_l, \left(\sum_\alpha a_\alpha(z_i) b_\alpha(\bar z_{i+1})\right) (h_1 \otimes \ldots \otimes h_l)\right\rangle =
\\
\sum_\alpha \Big( (g_1 \star h_1 \star f_1)(\x) \cdot \ldots (g_i \star a_\alpha(z_i) \star h_i \star f_i)(\x) \cdot \hskip 1cm\\
(g_{i+1} \star h_{i+1} \star b_\alpha(\bar z_{i+1}) \star f_{i+1})(\x) \cdot \ldots \cdot (g_l \star h_l \star f_l)(\x) \Big)=\\
\left\langle \sum_\alpha u_1 \otimes \ldots \hat u_{i \alpha} \otimes \hat u_{i+1 \alpha} \otimes \ldots\otimes u_l, h_1 \otimes  \ldots \otimes h_l \right\rangle,
\end{eqnarray*}
where $\hat u_{i \alpha}=\lambda(f_i \otimes (g_i \star a_\alpha))$ and $\hat u_{i+1\alpha}=\lambda((b_\alpha \star f_{i+1}) \otimes g_{i+1})$. We have thus proved that
\[
     \left(\nu\frac{\p \Phi}{\p z^p}(z_i, \bar z_{i+1}) \right)^t (u_1 \otimes \ldots \otimes u_l) =
     \sum_\alpha u_1 \otimes \ldots \hat u_{i \alpha} \otimes \hat u_{i+1 \alpha} \otimes \ldots\otimes u_l.
 \]
Now,
\begin{eqnarray*}
 W^{(l)}\left(\left(\nu\frac{\p \Phi}{\p z^p}(z_i, \bar z_{i+1}) \right)^t(u_1  \otimes \ldots \otimes u_m)\right) =\hskip 3.2cm \\
 W^{(l)}\left(\sum_\alpha u_1 \otimes \ldots \hat u_{i \alpha} \otimes \hat u_{i+1 \alpha} \otimes \ldots\otimes u_m\right)=\hskip 3cm\\
 \sum_\alpha \big((g_1 \star f_2)(\x) \cdot  \ldots \cdot (g_{i-1}\star f_i)(\x) \cdot 
 (g_i \star a_\alpha \star b_\alpha \star f_{i+1})(\x) \cdot \\
  (g_{i+1}\star f_{i+2})(\x) \cdot
 \ldots \cdot (g_l \star f_1)(\x)\big ).
 \end{eqnarray*}
 We see that
 \[
     \sum_\alpha (g_i \star a_\alpha \star b_\alpha \star f_{i+1})(\x) = \left(g_i \star \nu\frac{\p \Phi}{\p z^p} \star f_{i+1}\right)(\x),
 \]
because $a_\alpha \star b_\alpha = a_\alpha b_\alpha$. Hence,
 \begin{eqnarray*}
  W^{(l)}\left(\left(\nu\frac{\p \Phi}{\p z^p}(z_i, \bar z_{i+1})\right)^t (u_1 \otimes  \ldots \otimes u_l)\right) =\hskip 1.5cm\\
  W^{(l)}\left (\left(\nu \frac{\p}{\p z_i^p} + \nu \frac{\p \Phi}{\p z^p}(z_i, \bar z_i)\right)^t (u_1 \otimes  \ldots \otimes u_l)\right),
\end{eqnarray*}
which proves the first equality in (\ref{E:annih}). The second one can be checked similarly. 
\end{proof}
We have thus proved that formula (\ref{E:main}) allows to express the jet of $\exp G^{(l)}$ at $(\x)^l$ in terms of the algebra $(\Ncal, \bullet)$ for every $l \geq 1$.

\end{document}